\documentclass[a4paper, leqno, 12pt]{amsart}

\usepackage{txfonts}

\numberwithin{equation}{section}
\DeclareMathOperator*{\essinf}{ess\,inf}
\makeatletter
\renewcommand{\@biblabel}[1]{#1\hfill \hspace{-0.2cm}}
\makeatother

\usepackage[utf8]{inputenc}
\usepackage{amsmath}
\usepackage{amsfonts}
\usepackage{amssymb}
\usepackage[T1]{fontenc} 
\usepackage{url} 
\usepackage{color,esint,enumitem}
\usepackage[dvipsnames]{xcolor}
\usepackage[colorlinks, allcolors=RedViolet,pdfstartview=,pdfpagemode=UseNone]{hyperref} 
\usepackage[a4paper,top=2.54cm,bottom=2.54cm,left=1.90cm,right=1.90cm,%
                headsep=1em,includehead,includefoot]{geometry}

\usepackage{amsthm}
\theoremstyle{plain}
\newtheorem{thm}[equation]{Theorem}
\newtheorem{lem}[equation]{Lemma}

\newtheorem{cor}[equation]{Corollary}

\theoremstyle{definition}
\newtheorem{defn}[equation]{Definition}

\newtheorem{eg}[equation]{Example}

\theoremstyle{remark}
\newtheorem{rem}[equation]{Remark}

\numberwithin{equation}{section}

\newcommand{\R}{\mathbb{R}}

\newcommand{\Rn}{\mathbb{R}^n}

\renewcommand{\phi}{\varphi}
\renewcommand{\epsilon}{\varepsilon}
\def\le{\leqslant}
\def\leq{\leqslant}
\def\ge{\geqslant}
\def\geq{\geqslant}
\def\phi{\varphi}
\def\rho{\varrho}
\def\vartheta{\theta}

\newcommand{\Phiw}{\Phi_{\text{\rm w}}}
\newcommand{\Phic}{\Phi_{\text{\rm c}}}

\def\supp{\operatorname{supp}}
\def\esssup{\operatornamewithlimits{ess\,sup}}
\def\osc{\operatornamewithlimits{osc}}

\def\dist{\qopname\relax o{dist}}
\def\osc{\operatornamewithlimits{osc}}
\def\essinf{\operatornamewithlimits{ess\,inf}}

\def\supp{\operatornamewithlimits{supp}}

\def\loc{{\rm loc}}
\newcommand{\px}{{p(\cdot)}}
\newcommand{\qx}{{q(\cdot)}}

\newcommand{\inc}[1]{\hyperref[def:aInc]{{\normalfont(Inc){\ensuremath{_{#1}}}}}}
\newcommand{\dec}[1]{\hyperref[def:aDec]{{\normalfont(Dec){\ensuremath{_{#1}}}}}}
\newcommand{\ainc}[1]{\hyperref[def:aInc]{{\normalfont(aInc){\ensuremath{_{#1}}}}}}
\newcommand{\adec}[1]{\hyperref[def:aDec]{{\normalfont(aDec){\ensuremath{_{#1}}}}}}
\newcommand{\adeci}[1]{\hyperref[def:aDeci]{{\normalfont(aDec){\ensuremath{_{#1}^\infty}}}}}
\newcommand{\azero}{\hyperref[def:a0]{{\normalfont(A0)}}}
\newcommand{\aone}{\hyperref[def:a1]{{\normalfont(A1)}}}
\newcommand{\aones}[1]{\hyperref[def:a1s]{{\normalfont(A1-{\ensuremath{{#1}})}}}}

\date{\today}

\begin{document}

\title
{Bloch estimates in non-doubling generalized Orlicz spaces}

\author{Petteri Harjulehto}
\address{Petteri Harjulehto,
Department of Mathematics and Statistics,
FI-00014 University of Helsinki, Finland}
\email{\texttt{petteri.harjulehto@helsinki.fi}}

\author{Peter Hästö}
\address{Peter Hästö, Department of Mathematics and Statistics,
FI-20014 University of Turku, Finland}
\email{\texttt{peter.hasto@utu.fi}}

\author{Jonne Juusti}
\address{Jonne Juusti,
Department of Mathematics and Statistics,
FI-20014 University of Turku, Finland}
\email{\texttt{jthjuu@utu.fi}}


\begin{abstract}
We study minimizers of non-autonomous functionals 
\begin{align*}
\inf_u \int_\Omega \varphi(x,|\nabla u|) \, dx
\end{align*}
when $\varphi$ has generalized Orlicz growth. 
We consider the case where the upper growth rate of $\varphi$ is unbounded and
prove the Harnack inequality for minimizers. Our technique is based on ``truncating'' the function 
$\varphi$ to approximate the minimizer and Harnack estimates with uniform constants via a Bloch estimate 
for the approximating minimizers.
\end{abstract}

\maketitle


\begin{center}
\smallskip
\textit{Dedicated to Giuseppe Mingione on his  50th anniversary.}
\smallskip
\end{center}

\section{Introduction}

Minimizers of the variable exponent energy $\int |\nabla u|^{p(x)}dx$ have been studied in 
hundreds of papers. In almost all cases, it is assumed that there exist constants $c,C\in (1,\infty)$ 
such that $c\le p(x)\le C$ for all $x$. However, it is possible to use limiting procedures 
to study the borderline cases when $p(x)=1$ or $p(x)=\infty$ for some points \cite{HarHL08, HarHL09}. 
In recent years, minimizers of non-autonomous functionals 
\begin{align*}
\inf_u \int_\Omega \phi(x,|\nabla u|) \, dx
\end{align*}
have been studied when $\phi$ has generalized Orlicz growth
with tentative applications to anisotropic materials \cite{Zhi86} and image processing \cite{HarH21}. Again, the upper and lower growth rates are usually assumed to lie in $(1,\infty)$. 
In this article we consider the case when the upper growth rate is allowed to equal 
$\infty$ in some points and the lower growth rate is greater than $n$, the dimension. 
We prove the Harnack inequality for minimizers of such energies. 

Let us recall some information of the context by way of motivation. 
PDE with generalized Orlicz growth have been studied in many papers lately, 
both in the general setting and in particular special cases, 
such as the double phase case (e.g.\ \cite{BarCM15, BarCM18, ColM15a, ColM15b, Pap22}), 
perturbed variable exponent \cite{Ok16}, 
Orlicz variable exponent \cite{GPRT17}, 
degenerate double phase \cite{BCM2}, 
Orlicz double phase \cite{BaaB21, BOh3}, 
variable exponent double phase \cite{CBGHW22, MizNOS20, MizOS21}, 
multiple-phase \cite{BaaBO21, DeFO19}, and
double variable exponent \cite{ZhaR18}. 
Our framework includes all these cases. 

In the generalized Orlicz case it is known that 
solutions with given boudary values exist \cite{ChlGZ19,GwiSZ18, HarHK16}, 
minimizers or solutions with given boundary values are locally bounded, satisfy Harnack's inequality and belong to $C^{0, \alpha}_{\loc}$ \cite{BenK_pp, HarHL21, HarHT17, WanLZ19} or $C^{1, \alpha}_{\loc}$ \cite{HasO22, HasO_pp}, 
quasiminimizers satisfy a reverse H\"older inequality \cite{HarHK18}, 
minimizers for the obstacle problem are continuous \cite{Kar19} and the boundary Harnack inequality 
holds for harmonic functions \cite{ChlZ_pp}. Some articles deal with the non-doubling 
\cite{ChlGZ19b} or parabolic \cite{SkrV21} case as well as with the Gauss image problem \cite{LiSYY22}. 
We refer to the surveys \cite{Chl18, MinR21} and monographs \cite{ChlGSW21, HarH19, LanM19} for an 
overview. Advances have also been made in the field of $(p, q)$-growth problems 
\cite{DeFM20, DeFM21a, DeFM21b, Mar20, Mar21, MinP20}. 

In \cite{BenHHK21, BenK_pp}, the Harnack inequality was established in the 
doubling generalized Orlicz case for bounded or general solutions. 
In the current paper, we consider the effect of removing the 
assumption that the growth function is doubling thus allowing the upper growth rate to equal 
$\infty$. The approach is based on ideas from \cite{HarHL09, Lat04} involving approximating the energy functional. 
This is more difficult compared to the variable exponent case, since the form of the approximating 
problem is unclear as is the connection between solutions and minimizers. Additionally, 
the challenge in taking limits without the doubling assumption is to track the 
dependence of various constants on the parameters and to ensure that no extraneous 
dependence is introduced in any step. Nevertheless, we improve even the result for the 
variable exponent case.

Let us consider an example of our main result, Theorem~\ref{thm:harnack-2}. 
In the variable exponent case $\phi(x,t):=t^{p(x)}$ we compare with our 
previous result \cite[Theorem~6.4]{HarHL09}. In the previous result, we 
assumed that $\frac1p$ is Lipschitz continuous, but now we only need the 
more natural $\log$-Hölder continuity. 
Furthermore, the previous result applied only to small balls in which the exponent was 
(locally) bounded. The next example shows that the new result applies even to some sets 
where the the exponent is unbounded.

\begin{eg}[Variable exponent]\label{eg:varExp}
Define $p:B_1\to (n,\infty]$ on the unit ball $B_1$ as $p(x) := 2n\log\frac e{|x|}$. 
Hence $p(0)=\infty$ but $p<\infty$ a.e.
Assume that $f\in W^{1,\px}(B_1)$ with $\rho_\px(|\nabla f|)<\infty$.
If $u\in f+W^{1,\px}_0(B_1)$ is a minimizer of the $\px$-energy, then the Harnack inequality
\[
\sup_{B_r}(u+r)
\le C\inf_{B_r} (u+r)
\]
holds for $r\le \frac14$. 
The constant $C$ depends only on $n$ and $\rho_\px(|\nabla f|)$. 
Note that $B_r$ we have, in the notation of Theorem~\ref{thm:harnack-2}, 
$p^-=2n\log\frac er$ and $q^\circ=2n\log\frac {2e}r$ so that 
$\frac{q^\circ}{p^-} = \frac{1+\log2+\log \frac1r}{\log2+\log \frac1r}$ is bounded 
independent of $r$.
\end{eg}

In the double phase case we also obtain a corollary of Theorem~\ref{thm:harnack} which improves earlier results 
in that the dependence of the constant is only on $\frac qp$, not $p$ and $q$. Note that the usual 
assumption of Hölder continuity of $a$ is a special case of the 
inequality in the lemma, see \cite[Proposition~7.2.2]{HarH19}. 
Also note that the ``$+r$'' in the Harnack inequality is not needed in this case, since the double phase 
functional satisfies \aone{} in the range $[0,\frac K{|B|}]$ rather than 
$[1,\frac K{|B|}]$.

\begin{cor}[Double phase]\label{cor:doublePhase}
Let $\Omega\subset \Rn$ be a bounded domain, $n<p<q$ and $H(x,t):=t^p+a(x)t^q$.
Assume that $f\in W^{1,H}(\Omega)$ and 
\[
a(x)\lesssim a(y) + |x-y|^\alpha
\qquad\text{with}\quad 
\tfrac qp\le 1+\tfrac \alpha n
\]
for every $x,y\in \Omega$. 
Then any minimizer $u$ of the $\phi$-energy with boundary value function $f$ satisfies the Harnack inequality
\[
\sup_{B_r}u
\le C\inf_{B_r} u.
\]
The constant $C$ depends only on $n$, $\frac qp$ and $\rho_H(|\nabla f|)$.
\end{cor}


\section{Definitions and notation}

We briefly introduce our definitions. More information on $L^\phi$-spaces can be found in \cite{HarH19}. We assume that $\Omega \subset \Rn$ is a bounded domain, $n\ge 2$.
\textit{Almost increasing} means that there exists a constant $L\ge 1$ such that $f(s) \le L f(t)$ for all $s<t$. 
If there exists a constant $C$ such that $f(x) \le C g(x)$ for almost every $x$, then we write $f \lesssim g$. If $f\lesssim g\lesssim f$, then we write $f \approx g$. 

\begin{defn}
We say that $\phi: \Omega\times [0, \infty) \to [0, \infty]$ is a 
\textit{weak $\Phi$-function}, and write $\phi \in \Phiw(\Omega)$, if 
the following conditions hold for a.e.\ $x \in \Omega$:
\begin{itemize}
\item 
$y \mapsto \phi(y, f(y))$ is measurable for every measurable function $f:\Omega\to \R$.
\item
$t \mapsto \phi(x, t)$ is non-decreasing.
\item 
$\displaystyle \phi(x, 0) = \lim_{t \to 0^+} \phi(x,t) =0$ and $\displaystyle \lim_{t \to \infty}\phi(x,t)=\infty$.
\item 
$t \mapsto \frac{\phi(x, t)}t$ is $L$-almost increasing on $(0,\infty)$ with
constant $L$ independent of $x$.
\end{itemize}
If $\phi\in\Phiw(\Omega)$ is additionally convex and left-continuous with respect to $t$ for almost every $x$, then $\phi$ is a 
\textit{convex $\Phi$-function}, and we write $\phi \in \Phic(\Omega)$. If $\phi$ does not depend on $x$, then we omit the set and write $\phi \in \Phiw$ or $\phi \in \Phic$.

For $\phi\in\Phiw(\Omega)$ and $A\subset \Rn$ we denote $\displaystyle\phi^+_{A}(t) := \esssup_{x \in A\cap \Omega} \phi(x,t)$ and 
$\displaystyle\phi^-_{A}(t) := \essinf_{x \in A\cap \Omega} \phi(x,t)$. 
\end{defn}



We next define the un-weightedness condition \azero{}, the almost continuity conditions \aone{} 
and the growth conditions \ainc{} and \adec{}. Note that the constants $L_p$ and $L_q$ are 
independent of $x$ even though $p$ and $q$ can be functions. 

\begin{defn}
Let $s>0$, $p,q:\Omega\to[0,\infty)$ and let $\omega:\Omega\times [0,\infty)\to [0,\infty)$ be almost increasing with respect to the second variable. 
We say that $\phi:\Omega\times [0,\infty)\to [0,\infty)$ satisfies 
\begin{itemize}[leftmargin=4.5em]
\item[(A0)]\label{def:a0}
if there exists $\beta \in(0, 1]$ such that $ \phi(x,\beta)\le 1 \le \phi(x,\frac1{\beta})$ for a.e.\ $x \in \Omega$;
\item[(A1-$\omega$)] 
if for every $K \ge 1$ there exists $\beta \in (0,1]$ such that, for every ball $B$,
\[ 
\phi_B^+(\beta t) \le \phi_B^-(t) \quad\text{when}\quad \omega_B^-(t) \in \bigg[1, \frac{K}{|B|}\bigg];
\]
\item[(A1-$s$)]\label{def:a1n}\label{def:a1s}
if it satisfies \aones{\omega} for $\omega(x,t):=t^s$;
 \item[(A1)]\label{def:a1}
if it satisfies \aones{\phi};
\item[(aInc)$_\px$] \label{def:aInc} if
$t \mapsto \frac{\phi(x,t)}{t^{p(x)}}$ is $L_p$-almost 
increasing in $(0,\infty)$ for some $L_p\ge 1$ and a.e.\ $x\in\Omega$;
\item[(aDec)$_\qx$] \label{def:aDec}
if
$t \mapsto \frac{\phi(x,t)}{t^{q(x)}}$ is $L_q$-almost 
decreasing in $(0,\infty)$ for some $L_q\ge 1$ and a.e.\ $x\in\Omega$.
\end{itemize} 
We say that \ainc{} holds if \ainc{p} holds for some constant $p>1$, and similarly for \adec{}.
If in the definition of \ainc{\px} we have $L_p=1$, then we say that $\phi$ satisfies \inc{\px}, 
similarly for \dec{\qx}.
\end{defn}

Note that if $\phi$ satisfies \ainc{p} with a constant $L_p$, then it satisfies \ainc{r} for every 
$r\in (0, p)$ with the same constant $L_p$. This is seen as follows, with $s<t$:
\[
\frac{\phi(x,s)}{s^r}
= s^{p-r}\frac{\phi(x,s)}{s^{p}}
\le s^{p-r}L_p\frac{\phi(x,t)}{t^{p}}
= L_p\Big{(}\frac{s}{t}\Big{)}^{p-r}\frac{\phi(x,t)}{t^r}
\le L_p\frac{\phi(x,t)}{t^r}.
\]

Condition \aone{} with $K=1$ was studied in \cite{HarH19} under the name (A1$'$).
The condition \aones{\omega} was introduced in \cite{BenHHK21} 
to combine \aone{} and \aones{n} as well as other cases. It is the 
appropriate assumption if we have \textit{a priori} information that the 
solution is in $W^{1,\omega}$ or the corresponding Lebesgue or H\"older space. 
The most important cases are $\omega=\phi$ and $\omega(x,t)=t^s$, 
that is \aone{} and \aones{s}.


\begin{defn}\label{def:Lphi}
Let $\phi \in \Phiw(\Omega)$ and define the \textit{modular} 
$\varrho_\phi$ for $u\in L^0(\Omega)$, the set of measurable functions in $\Omega$, by 
 \begin{align*}
 \varrho_\phi(u) &:= \int_\Omega \phi(x, |u(x)|)\,dx.
 \end{align*}
The \emph{generalized Orlicz space}, also called Musielak--Orlicz space, is defined as the set 
 \begin{align*}
 L^\phi(\Omega) &:= \big\{u \in L^0(\Omega) :
 \lim_{\lambda \to 0^+} \varrho_\phi(\lambda u) = 0\big\}
 \end{align*}
equipped with the (Luxemburg) quasinorm 
 \begin{align*}
 \|u\|_{L^\phi(\Omega)} &:= \inf \Big\{ \lambda>0 : \varrho_\phi\Big(
 \frac{u}{\lambda} \Big) \le 1\Big\}.
 \end{align*}
We abbreviate $\|u\|_{L^\phi(\Omega)}$ by $\|u\|_{\phi}$ if the set is clear from context.
\end{defn}

\begin{defn}
A function $u \in L^\phi(\Omega)$ belongs to the
\textit{Orlicz--Sobolev space $W^{1, \phi}(\Omega)$} if its weak partial derivatives $\partial_1 u, \ldots, \partial_n u$ exist and belong to $L^{\phi}(\Omega)$.
For $u \in W^{1,\phi}(\Omega)$, we define the quasinorm
\[
\| u \|_{W^{1,\phi}(\Omega)} := \| u \|_{L^\phi(\Omega)} + \| \nabla u \|_{L^\phi(\Omega)}.
\]
We define \emph{Orlicz--Sobolev space with zero boundary values} $W^{1, \phi}_0(\Omega)$ as the closure of 
$\{u \in W^{1, \phi}(\Omega) : \supp u \subset \Omega\}$ in $W^{1, \phi}(\Omega)$. 
\end{defn}

In the definition $\| \nabla u \|_{L^\phi(\Omega)}$ is an abbreviation of $\big\| | \nabla u | \big\|_{L^\phi(\Omega)}$.
Again, we abbreviate $\| u \|_{W^{1,\phi}(\Omega)}$ by $\| u \|_{1,\phi}$ if $\Omega$ is clear from context.
$W^{1, \phi}_0(\Omega)$ is a closed subspace of $W^{1, \phi}(\Omega)$, and hence reflexive when 
$W^{1, \phi}(\Omega)$ is reflexive. We write $f + W^{1, \phi}_0(\Omega)$ to denote the set 
$\{f + v: v \in W^{1, \phi}_0(\Omega)\}$.

\begin{defn}\label{defn:minimizer}
We say that $u \in W^{1, \phi}_\loc(\Omega)$ is a \emph{local minimizer} if
\[
\int_{\supp h} \phi(x,|\nabla u|) \,dx
\le
\int_{\supp h} \phi(x,|\nabla(u + h)|) \,dx
\]
for every $h \in W^{1,\phi}(\Omega)$ with $\supp h \Subset \Omega$.
We say that $u \in W^{1, \phi}(\Omega)$ is a \emph{minimizer of the $\phi$-energy} with 
boundary values $f \in W^{1, \phi}(\Omega)$, if
$u-f \in W^{1, \phi}_0(\Omega)$, and 
\[
\int_{\Omega} \phi(x, |\nabla u|) \, dx \le \int_{\Omega} \phi(x, |\nabla v|) \, dx 
\] 
for every $v \in f+ W^{1, \phi}_0(\Omega)$. 
\end{defn}

Let $h \in W^{1, \phi}(\Omega)$ have compact support in $\Omega$, $f\in W^{1,\phi}(\Omega)$ and 
$u \in f+W^{1, \phi}_0(\Omega)$ is a minimizer of the $\phi$-energy. 
Then $u+h \in f + W^{1, \phi}_0(\Omega)$ by the definition. By the $\phi$-energy 
minimizing property, 
\[
\int_{\Omega} \phi(x, |\nabla u|) \, dx \le \int_{\Omega} \phi(x, |\nabla (u+h)|) \, dx;
\]
the integrals over the set $\Omega\setminus \supp h$ cancel, and so $u$ is a local minimizer. 
Hence every minimizer $u \in W^{1, \phi}(\Omega)$ of the $\phi$-energy with boundary values $f$ 
is a local minimizer.


\section{Auxiliary results}

We denote by $\phi^*$ the conjugate $\Phi$-function of $\phi\in\Phiw(\Omega)$, defined by 
\[
\phi^*(x,t):= \sup_{s\ge 0} (st-\phi(x,s)). 
\]
From this definition, we have Young's inequality $st \le \phi(x,s) + \phi^*(x,t)$. 
H\"older's inequality holds in generalized Orlicz spaces for $\phi\in\Phiw(\Omega)$ with constant $2$
\cite[Lemma~3.2.13]{HarH19}:
\[
\int_\Omega |u|\, |v|\, dx \le 2 \|u\|_\phi \|v\|_{\phi^*}.
\]
We next generalize the relation $\phi^*( \frac{\phi(t)}{t}) \le \phi(t)$ which 
is well-known in the convex case, to weak $\Phi$-functions. 
The next results are written for $\phi\in\Phiw$ but can be applied to 
$\phi\in\Phiw(\Omega)$ point-wise.

\begin{lem}\label{lem:phi*}
Let $\phi\in \Phiw$ satisfy \ainc{1} with constant $L$. Then 
\[
\phi^*\Big( \frac{\phi(t)}{Lt}\Big) \le \frac{\phi(t)}{L}.
\]
\end{lem}
\begin{proof}
When $s\le t$ we use $s\frac{\phi(t)}{Lt} - \phi(s) \le s\frac{\phi(t)}{Lt} \le \frac{\phi(t)}{L}$ 
to obtain 
\[
\phi^*\Big( \frac{\phi(t)}{Lt}\Big)
=
\sup_{s\ge 0} \Big(s\frac{\phi(t)}{Lt} - \phi(s) \Big)
\le 
\max \bigg\{\frac{\phi(t)}{L}, \sup_{s>t} \Big(s\frac{\phi(t)}{Lt} - \phi(s) \Big)\bigg\}.
\]
On the other hand, by \ainc{1} we conclude that $s\frac{\phi(t)}{Lt} \le \phi(s)$ when $s>t$, 
so the second term is non-positive and the inequality is established.
\end{proof}



If $\phi\in\Phiw$ is differentiable, then 
\begin{equation*} 
\frac d{dt}\frac{\phi(t)}{t^p} 
= 
\frac{\phi'(t) t^p- p t^{p-1}\phi(t)}{t^{2p}}
=
\frac{\phi(t)}{t^{p+1}}\bigg[ \frac{t\phi'(t)}{\phi(t)} - p \bigg].
\end{equation*}
Thus $\phi$ satisfies \inc{p} if and only if $\frac{t\phi'(t)}{\phi(t)} \ge p$.
Similarly, $\phi$ satisfies \dec{q} if and only if $\frac{t\phi'(t)}{\phi(t)} \le q$.
It also follows that if $\phi$ satisfies \inc{p} and \dec{q}, then 
\begin{equation}\label{eq:incdec} 
\tfrac{1}{q} t\phi'(t) \le \phi(t)\le \tfrac{1}{p} t\phi'(t)
\end{equation} and so
$\phi'$ satisfies \ainc{p-1} and \adec{q-1}. We next show that 
the last claim holds even if only \ainc{} or \adec{} is assumed of 
$\phi$ which is convex but not necessarily differentiable.

For $\phi \in \Phic$ we denote the left and right derivative by 
$\phi_-'$ and $\phi_+'$, respectively. We define the left derivative to be zero at the origin, i.e. $\phi_-'(0):=0$.
Assume that $\phi$ satisfies \ainc{p} with $p>1$, and let $t_0>0$ be such that $\phi(t_0) < \infty$. Then
\[
\phi'_+(0) = \lim_{t\to 0^+} \frac{\phi(t)}{t} \le \lim_{t\to 0^+} L_p t^{p-1} \frac{\phi(t_0)}{t_0^p} =0.
\]
Since $\phi_+'$ is right-continuous we also obtain that 
\[
\lim_{t\to 0^+}\phi_+'(t)
= \phi_+'(0)= 0.
\]

\begin{lem}\label{lem:aincadec}
Let $\phi \in \Phic$ satisfy \ainc{p} and \adec{q} with constants $L_p$ and $L_q$, respectively.
Then \[
 \frac{1}{(L_qe-1) q} t\phi'_+(t) \le \phi(t) \le \frac{2\ln(2L_p)}p t\phi'_-(t)
\]
for every $t\ge 0$, and $\phi'_-$ and $\phi'_+$ satisfy \ainc{p-1} and \adec{q-1}, with 
constants depending only on $\frac{q}{p}$, $L_p$ and $L_q$. 
\end{lem}

\begin{proof}
Since $\phi$ is convex we have
\[
\phi(t) 
= \int_0^t \phi'_+ (\tau) \, d \tau 
= \int_0^t \phi'_- (\tau) \, d \tau, 
\]
for a proof see e.g.\ \cite[Proposition 1.6.1, p.~37]{NicP06}.
Let $r \in [0,1)$. Since the left derivative is increasing, we obtain 
\[
\phi(t)-\phi(rt) = \int_{rt}^t \phi'_- (\tau) \, d \tau
\le (t- rt) \phi'_- (t).
\]
Thus
\[
t\phi'_- (t)
\geq \frac{\phi(t)-\phi(rt)}{1-r}
\geq \phi(t)\frac{1-L_pr^p}{1-r}
\]
where in the second inequality we used \ainc{p} of $\phi$.
Choosing $r := (2L_p)^{-1/p}$ we get
\[
\frac{1-L_pr^p}{1-r}
= \frac{1/2}{1-(2L_p)^{-1/p}}
= \frac{p}{2p(1-(2L_p)^{-1/p})}.
\]
Writing $h := \frac{1}{p}$ and $x:=2L_p$, we find that
\[
p\left(1-(2L_p)^{-1/p}\right)
=\frac{1-x^{-h}}{h}
\leq \left.\frac{dx^s}{ds}\right|_{s=0}
= \ln x,
\]
where the inequality follows from convexity of $s \mapsto x^s$.
Thus
$
t\phi'_- (t)
\geq \frac{p}{2\ln (2L_p)}\phi(t)
$.

Let $R > 1$. Since $\phi_+'$ is increasing, we obtain
\[ 
\phi(Rt)-\phi(t) = \int_{t}^{Rt} \phi'_+(\tau) \, d \tau 
\ge (Rt -t) \phi'_+ (t). 
\]
Thus
\[
t\phi'_+ (t)
\leq \frac{\phi(Rt)-\phi(t)}{R-1}
\leq \phi(t)\frac{L_qR^q-1}{R-1}
\]
where \adec{q} of $\phi$ was used in the second inequality.
With $R := 1+\frac{1}{q}$ we get
\[
t\phi'_+(t)
\leq \frac{L_q(1+\frac{1}{q})^q-1}{1/q} \phi(t)
\leq q(L_qe-1)\phi(t).
\]
We have established the inequality of the claim. 

We abbreviate $c_q:=\frac1{L_qe-1}$ and $c_p:=2\ln(2L_p)$.
Since $\phi$ is convex, we have $\phi_-' \le \phi_+'$ and so 
\[
\frac{c_q} q t\phi'_-(t) \le \frac{c_q} q t\phi'_+(t) \le \phi(t) \le \frac{c_p}p t\phi'_-(t) \le \frac{c_p}p t\phi'_+(t)
\]
Thus we obtain by \adec{q} of $\phi$ for $0<s<t$ that
\[
\frac{\phi'_+(t)}{t^{q-1}} \le \frac{q}{c_q} \frac{\phi(t)}{t^{q}} 
\le \frac{q L_q}{c_q} \frac{\phi(s)}{s^{q}} 
\le 
\frac qp\frac{L_q c_p}{c_q} \frac{\phi'_+(s)}{s^{q-1}} 
\]
and \adec{q-1} of $\phi'_+$ follows. The proof for \ainc{p-1} is similar as are the proofs for $\phi'_-$. 
\end{proof}

Before Lemma~\ref{lem:aincadec} we noted that $\lim_{t\to 0^+} \phi_+'(x,t) =0$, and hence
\[
\lim_{|y|\to 0} \frac{\phi_+'(x,|y|)}{|y|} y\cdot z
= \Big(\lim_{|y|\to 0^+} \phi_+'(x,|y|)\Big) \Big(\lim_{|y|\to 0^+} \frac{y}{|y|}\cdot z\Big)
= 0
\]
for $z\in\Rn$. In light of this, we define
\begin{equation*}
\frac{\phi_+'(x,|\nabla u|)}{|\nabla u|}\nabla u \cdot \nabla h
:= 0
\qquad\text{when }\nabla u = 0.
\end{equation*}

\begin{thm}\label{thm:almostSolutions}
Let $\phi \in \Phic(\Omega)$ satisfy \ainc{p} and \adec{q} with $1 < p \leq q$.
If $u \in W_{\loc}^{1,\phi}(\Omega)$, then the following are equivalent:
\begin{itemize}
\item[(i)] 
$u$ is a local minimizer;
\item[(ii)]
$\displaystyle
\int_{\supp h} \frac{\phi'_h(x,|\nabla u|)}{|\nabla u|} \nabla u \cdot \nabla h \,dx
\ge 0$ for every $h \in W^{1,\phi}(\Omega)$ with $\supp h \Subset \Omega$.
\end{itemize}
Here $\phi'_h := \phi_+'\chi_{\{\nabla u \cdot \nabla h \geq 0\}} + \phi_-'\chi_{\{\nabla u \cdot \nabla h < 0\}}$.
\end{thm}

\begin{proof}
Let $h \in W^{1,\phi}(\Omega)$ with $E := \supp h \Subset \Omega$ be arbitrary.
Define $g: \Omega \times [0,1] \to [0,\infty]$ by 
$g(x, \epsilon) := |\nabla (u(x) + \epsilon h(x))|$; in the rest of the proof we omit the first variable and 
abbreviate $g(x, \epsilon)$ by $g(\epsilon)$.

Note that $g(\epsilon)^2=|\nabla u(x)|^2+ \epsilon^2|h(x)|^2+2\epsilon \nabla u(x)\cdot \nabla h(x)$ 
and $g\ge 0$. 
Thus in $[0, 1]$ the function $g$ has a local minimum at zero for $x\in\Omega$ with $\nabla u(x)\cdot \nabla h(x)\ge 0$ and a  maximum
otherwise. This determines whether we obtain the right- or left-derivative
and so 
\begin{equation}\label{eq:pointwiseLim}
\lim_{\epsilon\to 0^+} \frac{\phi(x,g(\epsilon)) - \phi(x,g(0))}{\epsilon}
=
\phi'_h(x, g(0))g'(0)
= \frac{\phi'_h(x,|\nabla u|)}{|\nabla u|} \nabla u \cdot \nabla h
\end{equation}
for almost every $x \in E$.

Let us then find a majorant for the expression on the left-hand side of \eqref{eq:pointwiseLim}.
By convexity,
\[
\Big|\frac{\phi(x,g(\epsilon)) - \phi(x,g(0))}{\epsilon}\Big|
\le
\phi_+'(x,\max \{ g(\epsilon),g(0)\})\frac{|g(\epsilon)-g(0)|}{\epsilon}
\]
for a.e.\ $x\in E$.
Since $\epsilon \in [0, 1]$ we have
\[
\max \{ g(\epsilon),g(0)\}
\le \max\{|\nabla u| + \epsilon |\nabla h| , |\nabla u|\}
\leq |\nabla u| + |\nabla h|.
\]
By the triangle inequality, 
\[
\Big|\frac{g(\epsilon) - g(0)}{\epsilon}\Big|
=
\Big|\frac{|\nabla u + \epsilon\nabla h|-|\nabla u|}{\epsilon}\Big|
\le
|\nabla h|\le |\nabla u| + |\nabla h|.
\]
Combining the estimates above, we find that
\[
\Big|\frac{\phi(x,g(\epsilon)) - \phi(x,g(0))}{\epsilon}\Big|
\leq \phi_+'(x, |\nabla u| + |\nabla h|)(|\nabla u| + |\nabla h|).
\]
By Lemma~\ref{lem:aincadec}, $\phi_+'(x,t)t\lesssim\phi(x,t)$ for every $t \ge 0$, so that 
\[
\phi_+'(x, |\nabla u| +|\nabla h|)(\nabla u| + |\nabla h|)
\lesssim
\phi(x, |\nabla u| + |\nabla h|).
\]
By \adec{},
\[
\phi(x,|\nabla u| + |\nabla h|)
\leq \phi(x,2|\nabla u|) + \phi(x,2|\nabla h|)
\leq L_q2^q(\phi(x,|\nabla u|) + \phi(x,|\nabla h|))\quad\text{a.e.}
\]
Combining the estimates, we find that 
\[
\Big|\frac{\phi(x,g(\epsilon)) - \phi(x,g(0))}{\epsilon}\Big|
\lesssim
\phi(x,|\nabla u|) + \phi(x,|\nabla h|) \quad\text{a.e.}
\]
The right hand side is integrable by \cite[Lemma~3.1.3(b)]{HarH19}, since $|\nabla u|,|\nabla h| \in L^\phi(\Omega)$ and $\phi$ satisfies \adec{}.
Thus we have found a majorant.
By dominated convergence and \eqref{eq:pointwiseLim}, we find that
\begin{equation}\label{eq:solutionLim}
\int_E \frac{\phi'_h(x,|\nabla u|)}{|\nabla u|} \nabla u \cdot \nabla h \,dx
= \lim_{\epsilon\to 0^+} \int_E \frac{\phi(x,g(\epsilon)) - \phi(x,g(0))}{\epsilon} \,dx.
\end{equation}
Let us first show that (i) implies (ii). By (i), 
\begin{equation*}
\int_E \frac{\phi(x,g(\epsilon)) - \phi(x,g(0))}{\epsilon} \,dx
\geq 0
\end{equation*}
for $\epsilon \in (0,1]$, and hence (ii) follows by \eqref{eq:solutionLim}.

Let us then show that (ii) implies (i).
For $\theta \in [0,1]$ and $s,t \ge 0$ we have
\[
\begin{split}
g(\theta t + (1-\theta)s) &= |\theta \nabla u + \theta t \nabla h + (1-\theta) \nabla u + (1-\theta) s \nabla h |\\
&\le |\theta \nabla u + \theta t \nabla h |+ |(1-\theta) \nabla u + (1-\theta) s \nabla h | = \theta g(t) + (1-\theta) g(s),
\end{split}
\]
so $g(\epsilon)$ is convex. 
Since $t \mapsto \phi(x,t)$ and $g(\epsilon)$ are convex for almost every $x \in E$, and $t \mapsto \phi(x,t)$ is also increasing, the composed function $t \mapsto \phi(x,g(t))$ is convex for a.e.\ $x \in E$.
Thus
\[
\int_E \phi(x,g(1)) - \phi(x,g(0)) \,dx
\geq \int_E \frac{\phi(x,g(\epsilon)) - \phi(x,g(0))}{\epsilon} \,dx.
\]
Since the above inequality holds for every $\epsilon \in (0,1)$, \eqref{eq:solutionLim} implies that
\[
\int_E \phi(x,g(1)) - \phi(x,g(0)) \,dx
\geq \int_E \frac{\phi'_h(x,|\nabla u|)}{|\nabla u|} \nabla u \cdot \nabla h \,dx \ge 0,
\]
which is (i).
\end{proof}

We conclude the section by improving the Caccioppoli inequality from \cite{BenHHK21}; 
in this paper we only need the special case $\ell=1$ and $s=q$, but we include the 
general formulation for possible future use. 
We denote by $\eta$ a cut-off function in $B_R$, more precisely, 
$\eta \in C_0^{\infty}(B_R)$, $\chi_{B_{\sigma R}} \le \eta \le \chi_{B_R}$ 
and $|\nabla \eta| \le \frac{2}{(1-\sigma)R}$, where $\sigma \in (0, 1)$. 
Note that the auxiliary function 
$\psi$ is independent of $x$ in the next lemma. Later on we will 
choose $\psi$ to be a regularized version of $\phi_B^+$. 
Note also that the constant in the lemma is independent of $q_1$.

\begin{lem}[Caccioppoli inequality]
\label{lem:caccioppoli}
Suppose $\phi \in \Phic(\Omega)$ satisfies \ainc{p} and \adec{q} with constants $L_p$ and $L_q$, 
and let $\psi\in \Phiw$ be differentiable and satisfy \azero{}, 
\inc{p_1} and \dec{q_1}, $p_1, q_1\ge 1$. Let $\beta\in (0, 1]$ be the constant from \azero{} of $\psi$. 
If $u$ is a non-negative local minimizer and $\eta$ is a cut-off function in $B_R\subset \Omega$, then 
\begin{align*}
\int_{B_R} \phi(x, |\nabla u|) \psi( \tfrac{u+R}{\beta R} )^{-\ell} \eta^s \, dx 
\le 
K \int_{B_R} \psi (\tfrac{u+R}{\beta R} )^{-\ell} \phi\big(x,K\, \tfrac{u+R}{\beta R} \big) \eta^{s-q} \, dx
\end{align*}
for any $\ell > \frac1{p_1}$ and $s\ge q$, 
where $\displaystyle K := \frac{8s q (L_qe-1)L_q\ln(2L_p)}{p (p_1\ell- 1)(1-\sigma)} + L_p$.
\end{lem}
\begin{proof}
Let us simplify the notation by writing $\tilde u := u + R$ and $v := \frac{\tilde u}{\beta R}$. 
Since $\nabla u = \nabla \tilde u$, we see that $\tilde u$ is still a local minimizer. 
By \azero{} of $\psi$ and $v\ge \frac1\beta$, we have $0\le \psi(v)^{-\ell} \le 1$.

We would like to use Theorem~\ref{thm:almostSolutions} with $h := \psi(v)^{-\ell} \eta^{s} \tilde u$. Let us first check that $h$ is a valid test function for a local minimizer, that is $h \in W^{1,\phi}(B_R)$ and has compact support in $B_R \subset \Omega$. As $\tilde u \in L^{\phi}(B_R)$ and $|h|\le \tilde u$, it is immediate that $h \in L^{\phi}(B_R)$. By a direct calculation,
\begin{align*}
\begin{split}
\nabla h 
&= -\ell \psi(v)^{-\ell-1} \eta^{s} \tilde u \psi'(v) \nabla v + 
s \psi(v)^{-\ell} \eta^{s-1} \tilde u \nabla \eta + \psi(v)^{-\ell} \eta^s \nabla \tilde u. 
\end{split}
\end{align*}
Note that $\tilde u \nabla v = v \nabla \tilde u$. Since $\psi$ is differentiable we may use \eqref{eq:incdec} to get
\begin{align*}
\big|\ell \psi(v)^{-\ell-1} \eta^{s} \psi'(v) v \nabla \tilde u \big|
\le \ell \psi(v)^{-\ell-1} q_1 \psi(v) | \nabla \tilde u| \le
q_1\ell |\nabla \tilde u| \in L^{\phi}(B_R).
\end{align*}
For the third term in $\nabla h$, we obtain $|\psi(v)^{-\ell} \eta^{s} \nabla \tilde u |
\le |\nabla \tilde u| \in L^{\phi}(B_R)$.
The term with $\nabla \eta$ is treated as $h$ itself. Thus $h \in W^{1, \phi}(B_R)$.
Since $s>0$ and $\eta\in C^\infty_0(B_R)$, $h$ has compact support in 
$B_R \subset \Omega$ and so it is a valid test-function for a local minimizer. 

We next calculate
\begin{align*}
\nabla \tilde u \cdot \nabla h 
&=-\psi(v)^{-\ell-1} \eta^{s} [\ell \psi'(v) v - \psi(v)] |\nabla \tilde u|^2
+ s\psi(v)^{-\ell} \eta^{s-1} \tilde u \, \nabla\tilde u \cdot \nabla \eta .
\end{align*}
The inequality $p_1 \psi(t) \le \psi'(t) t$ from \eqref{eq:incdec} 
implies that $\ell \psi'(v) v - \psi(v) \ge (p_1\ell- 1)\psi(v) > 0$. 
Since $\tilde u$ is a local minimizer, we can use the implication (i)$\Rightarrow$(ii) of Theorem~\ref{thm:almostSolutions} to conclude that 
\begin{align*}
[p_1\ell- 1] \int_{B_R} \phi'_h(x, |\nabla \tilde u|)|\nabla \tilde u| \psi(v)^{-\ell} \eta^{s}
\, dx \le
s \int_{B_R} \phi'_h(x, |\nabla \tilde u|) \psi(v)^{-\ell}\tilde u 
\,|\nabla \eta| \, \eta^{s-1} \, dx.
\end{align*}
Since $\phi'_-\le\phi'_h\le\phi'_+$, we obtain 
$\frac{1}{q(L_qe-1))} t\phi'_h(x,t) \le \phi(x,t) \le \frac{2 \ln(2L_p)}p t\phi'_h(x,t)$ 
from Lemma~\ref{lem:aincadec}. Using also $|\nabla \eta|\,\tilde u\le \frac 2{1-\sigma}v$, we have
\begin{align*}
& \int_{B_R} \phi(x, |\nabla \tilde u|) \psi(v)^{-\ell} \eta^{s}\, dx 
\le
\frac{4 s q (L_qe-1)\ln(2L_p)}{p (p_1\ell- 1)(1-\sigma)} \int_{B_R} \frac{\phi(x, |\nabla \tilde u|)}{|\nabla \tilde u|}\, \eta^{s-1}\psi(v)^{-\ell}v \, dx,
\end{align*}
Note that the constant in front of the integral can be estimated from above by 
$\frac K{2L_q}$. 

Next we estimate the integrand on the right hand side. 
By Young's inequality 
\begin{equation*}
\frac{\phi(x, |\nabla \tilde u|)}{|\nabla \tilde u|} v 
\le 
\phi\big(x, \epsilon^{-\frac1{q'}}L_p v\big)
+ \phi^*\big(x, \epsilon^{\frac1{q'}} L_p^{-1}\tfrac{\phi(x, |\nabla \tilde u|)}{|\nabla \tilde u|}\big),
\end{equation*}
where $\frac1q + \frac1{q'}=1$.
We choose $\epsilon:=\frac{L_p}{K}\eta(x)\in (0, 1]$ and use
\ainc{q'} of $\phi^*$ \cite[Proposition~2.4.9]{HarH19} (which holds with constant $L_q$) and 
Lemma~\ref{lem:phi*} to obtain
\[
\begin{split}
\phi^*\big(x, \epsilon^{\frac1{q'}} L_p^{-1}\tfrac{\phi(x, |\nabla \tilde u|)}{|\nabla \tilde u|}\big)
\le
L_q \epsilon \phi^*\big(x, \tfrac{\phi(x, |\nabla \tilde u|)}{L_p |\nabla \tilde u|}\big)
&\le
 \tfrac{L_q \epsilon}{L_p}\phi(x,|\nabla u|)
= \tfrac{L_q}{K}\eta(x) \phi(x,|\nabla u|).
\end{split}
\]
In the other term we estimate $\epsilon^{-\frac1{q'}}L_p \le K^{1-\frac1{q'}} L_p^{\frac1{q'}} \epsilon^{-\frac1{q'}} = \eta^{-\frac1{q'}}K$
and use \adec{q} of $\phi$:
\[
\phi\big(x, \epsilon^{-\frac1{q'}}L_p v\big)
\le
L_q \eta^{1-q} \phi\big(x, K v\big).
\]
With these estimates we obtain that
\begin{align*}
&\int_{B_R} \!\phi(x, |\nabla \tilde u|) \psi(v)^{-\ell} \eta^{s} \, dx
\le \frac{1}{2}\! \int_{B_R} \!\phi(x, |\nabla \tilde u|) \psi(v)^{-\ell} \eta^{s} \, dx 
 + \frac K2 \! \int_{B_R}\! \psi(v)^{-\ell} \phi(x, Kv) \eta^{s-q} \, dx.
\end{align*}
The first term on the right-hand side can be absorbed in the left-hand side. This gives the claim. 
\end{proof}

The next observation is key to applications with truly non-doubling growth. 

\begin{rem}\label{rem:caccioppoliHole}
In the previous proof the assumption \adec{q} is only needed in the set 
$\nabla \eta\ne 0$ since we can improve the estimate on the right-hand side integral 
to $|\nabla \eta|\,\tilde u\le \frac 2{1-\sigma}v \chi_{\{\nabla \eta\ne 0\}}$ and only 
drop the characteristic function in the final step.
\end{rem}

\section{Bloch-type estimate for bounded supersolutions}

The following definition is like \cite[Definition~3.1]{BenHHK21}, except 
$\phi_{B_r}^+$ has replaced $\phi_{B_r}^-$. Furthermore, we are more precise with our 
estimates so as to avoid dependence on $p$ and $q$. 

\begin{defn}\label{def:psiR}
Let $\phi \in \Phiw(B_r)$ satisfy \ainc{p} with $p\ge 1$ and constant $L_p$.
We define $\psi_{B_r}: B_r \to [0, \infty]$ by setting 
\[
\psi_{B_r}(t)
:= \int_0^t \tau^{p-1} \sup_{s \in (0,\tau]} \frac{\phi_{B_r}^+(s)}{s^p} \,d\tau \quad\text{for }t \ge 0.
\]
\end{defn}
It is easy to see that $\psi_{B_r}\in \Phiw$.
Using that $\phi$ is increasing for the lower bound and \ainc{p} for the upper bound, 
we find that 
\begin{equation}\label{eq:psiPhiIneq}
\ln(2) \phi_{B_r}^+\Big{(}\frac{t}{2}\Big{)}
=
\int_{t/2}^t \tau^{p-1} \frac{\phi_{B_r}^+(t/2)}{\tau^p} \,d\tau
\le
\psi_{B_r}(t)
\le \int_0^t t^{p-1} L_p \frac{\phi_{B_r}^+(t)}{t^p} \,d\tau
= L_p\phi_{B_r}^+(t).
\end{equation}
As in \cite[Definition~3.1]{BenHHK21}, we see that $\psi_{B_r}$ is convex and satisfies \inc{p}.
If $\phi$ satisfies \azero{}, so does $\psi_{B_r}$, since $\phi^+_{B_r}\simeq\psi_{B_r}$.
If $\phi$ satisfies \adec{q}, then $\psi_{B_r}$ is strictly increasing and satisfies \adec{q}, 
and, as a convex function, also \dec{} \cite[Lemma~2.2.6]{HarH19}. 

We note in both the above reasoning and in the next theorem that constants 
have no direct dependence on $p$ or $q$, only on $L_p$, $L_q$ and $\frac qp$. 

\begin{thm}[Bloch-type estimate]\label{thm:Bloch}
Let $\phi \in \Phic(\Omega)$ satisfy \azero{} and \aone{}.
Let $B_{2r}\subset \Omega$ with $r\le 1$ and $\phi|_{B_{2r}}$ satisfy 
\ainc{p} and \adec{q} with $p,q\in[n, \infty)$.
If $u$ is a non-negative local minimizer, then 
\[
\int_{B_r} |\nabla\log(u+r)|^n \,dx
\le C,
\]
where $C$ depends only on $n$, 
$L_p$, $L_q$, $\frac qp$, the constants from \azero{} and \aone{}, and $\rho_\phi(|\nabla u|)$.
\end{thm}

\begin{proof}
Let us first note that $\phi$ satisfies \ainc n with the constant $L_p$.
Let $\beta$ be the smaller of the constants from \azero{} and \aone{}. 
Denote $v := \frac{u+2r}{2\beta r}$ and $\gamma:= \frac{2K}\beta$, 
where $K$ is from Caccioppoli inequality (Lemma~\ref{lem:caccioppoli}) with $\ell = 1$, $s = q$ and 
$\sigma = \frac{1}{2}$. Since $p\ge n$, we see that 
\[
K \le 
\frac{16 q^2 (L_qe-1)L_q\ln(2L_p)}{p(p-1)} + L_p
\le 
16 (L_qe-1)L_q\ln(2L_p) \Big(\frac{q}{p}\Big)^2 \frac{n}{n-1} + L_p. 
\]

When $|\nabla u| > \gamma v$, we use \ainc{n} to deduce that 
\[
\frac{\phi_{B_{2r}}^-(\gamma v)}{v^n}
\le \gamma^n\frac{\phi(x, \gamma v)}{(\gamma v)^n}
\le \gamma^n L_p\frac{\phi(x,|\nabla u|)}{|\nabla u|^n}
\]
for a.e.\ $x \in B_r$.
Rearranging gives $\frac{|\nabla u|^n}{v^n}\lesssim \frac{\phi(x,|\nabla u|)}{\phi_{B_{2r}}^-(\gamma v)}$. 
Since $v \ge \frac1\beta$ and $\gamma \ge 1$, we obtain by \azero{} that $\phi_{B_{2r}}^-(\gamma v)\ge 1$. If also $\phi_{B_{2r}}^-(\gamma v) \le \frac{1}{|B_{2r}|}$, then $\phi_{B_{2r}}^+(\beta \gamma v) \le \phi_{B_{2r}}^-(\gamma v)$ by \aone{}. Otherwise, 
$(\phi_{B_{2r}}^-(\gamma v))^{-1} \le |B_{2r}|$.
In either case, 
\[
\frac{|\nabla u|^n}{v^n}
\lesssim
\frac{\phi(x,|\nabla u|)}{\phi_{B_{2r}}^-(\gamma v)}
\lesssim 
\phi(x,|\nabla u|) \Big(\frac{1}{\phi_{B_{2r}}^+(\beta\gamma v)} + |B_{2r}|\Big)
\]
for a.e.\ $x\in B_r$. When $|\nabla u| \le \gamma v$, we use the estimate 
$\frac{|\nabla u|^n}{v^n}\le \gamma^n$ instead.
Since $u+r \geq \frac{1}{2}(u+2r) = \beta r v$, we obtain that
\[
\begin{split}
\int_{B_r} |\nabla\log(u+r)|^n \,dx
= \int_{B_r} \frac{|\nabla u|^{n}}{(u+r)^n} \,dx
\le \frac{1}{(\beta r)^n}\int_{B_r} \frac{|\nabla u|^{n}}{v^n} \,dx 
&\lesssim 
\fint_{B_r} \frac{\phi(x,|\nabla u|)}{\phi_{B_{2r}}^+(\beta\gamma v)} 
+ |B_{2r}|\, \phi(x,|\nabla u|) +1 \,dx\\
&= 
\fint_{B_r} \frac{\phi(x,|\nabla u|)}{\phi_{B_{2r}}^+(\beta\gamma v)} \,dx
+2^n\rho_\phi(|\nabla u|)+1.
\end{split}
\]
It remains to bound the integral on the right-hand side. 

Let $\psi_{B_{2r}}$ be as in Definition \ref{def:psiR}, let $\eta \in C_0^\infty(B_{2r})$ be 
a cut-off function such that $\eta = 1$ in $B_r$ and choose $\psi(t):=\psi_{B_{2r}}(\beta\gamma t)$. Then
\[
\fint_{B_r} \frac{\phi(x,|\nabla u|)}{\phi_{B_{2r}}^+(\beta\gamma v)} \,dx
\lesssim \fint_{B_{2r}} \frac{\phi(x,|\nabla u|)}{\phi_{B_{2r}}^+(\beta\gamma v)} \eta^q \,dx
\le L_p\fint_{B_{2r}} \frac{\phi(x,|\nabla u|)}{\psi(v)} \eta^q \,dx,
\]
where the second inequality follows from \eqref{eq:psiPhiIneq}.
We note that $\psi$ satisfies \azero, \inc{p} and \dec.
Now we use the Caccioppoli inequality (Lemma~\ref{lem:caccioppoli}) for $\phi$ and 
$\psi$ with $\ell = 1$, $s = q$ and $\sigma = \frac{1}{2}$ to get
\[
\fint_{B_{2r}} \frac{\phi(x,|\nabla u|)}{\psi(v)} \eta^q\,dx
\le K \fint_{B_{2r}} \frac{\phi(x,K v)}{\psi(v)} \,dx
\le \ln(2) K;
\]
the last inequality holds by \eqref{eq:psiPhiIneq} and $\gamma=\frac{2K}\beta$ since 
\[
\frac{\phi(x,K v)}{\psi(v)}
=
\frac{\phi(x,K v)}{\psi_{B_{2r}}(\beta\gamma v)}
\le
\ln(2) \frac{\phi(x,K v)}{\phi_{B_{2r}}^+(\frac12\beta\gamma v)}\le \ln(2). \qedhere
\] 
\end{proof}


We next show that the Bloch estimate implies a Harnack inequality for suitable monotone 
functions. We say that a continuous function $u$ is \textit{monotone in the sense of Lebesgue}, 
if it attains its extrema on the boundary of any compact set in its domain of definition.
We say that $\phi\in\Phiw(\Omega)$ is positive if 
$\phi(x,t)>0$ for every $t>0$ and a.e.\ $x\in\Omega$. If $\phi$ satisfies 
\adec{\qx} for $q<\infty$ a.e., then it is positive.

\begin{lem}\label{lem:monotone}
If $\phi\in\Phiw(\Omega)$ is positive, then every continuous local minimizer is monotone in the sense of Lebesgue.
\end{lem}
\begin{proof}
Let $u\in W_{\loc}^{1,\phi}(\Omega)\cap C(\Omega)$ be a local minimizer and $D\Subset\Omega$. 
Fix $M>\max_{\partial D} u$ and note that $(u-M)_+$ is zero in some neighbourhood of $\partial D$ since $u$ is continuous.
Thus $h:=(u-M)_+ \chi_D$ belongs to $W^{1,\phi}(\Omega)\cap C(\Omega)$ and has compact support in $\Omega$.
Using that $u$ is a local minimizer, we obtain that 
\[
\int_{\supp h}\phi(x,|\nabla u|)\,dx
\le 
\int_{\supp h}\phi(x,|\nabla (u-h)|)\,dx
=
0.
\]
Since $\phi(x,t)>0$ for every $t>0$ and a.e.\ $x\in\Omega$, it follows that $\nabla u=0$ a.e.\ 
in $\supp h$. Thus $\nabla h=0$ a.e.\ in $\Omega$. Since $h$ is continuous and equals $0$ in 
$\Omega\setminus D$, we conclude that $h\equiv 0$. 
%
Hence $u\le M$ in $D$.
Letting $M\to \max_{\partial D} u$, we find that $u\le \max_{\partial D} u$. 
The proof that $\min_{\partial D} u\le u$ in $D$ is similar.
\end{proof}

For $x \in \Omega$ we write $r_x := \frac12 \dist(x, \partial \Omega)$. Let $1 \le p < \infty$. In \cite[Definition 3.6]{Lat04} a function $u:\Omega \to \R$ is called a \emph{Bloch function} if 
\[
\sup_{x\in \Omega} r_x^p \fint_{B_{r_x}}  |\nabla u|^p \, dx < \infty.
\]
Note that if $u$ is an analytic function in the plane and $p=2$, then by the mean value 
property
\[
\sup_{x\in \Omega} r_x \Big(\fint_{B_{r_x}} |u'|^2 \, dx \Big)^\frac12
\approx 
\sup_{x\in \Omega} d(x,\partial \Omega)\,|u'(x)|
\approx 
\sup_{x\in \Omega} (1-|x|^2) |u'(x)|, 
\]
which connects this with Bloch functions in complex analysis. 
In the next theorem we assume for $\log u$  a Bloch-type condition.
In the case $p=n$ the next result was stated in \cite[Lemma~6.3]{HarHL09}.

\begin{lem}\label{lem:Harnack}
Let $u:\Omega\to (0,\infty)$ be continuous and monotone in the sense of Lebesgue.
If $B_{4r}\Subset\Omega$, and
\[
r^p \fint_{B_{2r}}|\nabla\log u|^p \, dx 
\le A
\]
for $p>n-1$, then
\[
\sup_{B_{r}} u
\le C\inf_{B_{r}}u
\]
for some $C$ depending only on $A$, $p$ and $n$.
\end{lem}

\begin{proof}
Denote $v:=\log u$.
Since the logarithm is increasing, $v$ is monotone in the sense of Lebesgue because $u$ is.

As $u$ is continuous and positive in $\overline{B_{3r}}\subset \Omega$, it is bounded away from $0$.
Thus $v\in W^{1,p}(B_{2r})$ is uniformly continuous in $B_{3r}$. Mollification gives a sequence 
$(v_i)_{i=0}^\infty$ of functions in $C^\infty(B_{2r})\cap W^{1,p}(B_{2r})$, such that $v$ 
is the limit of $v_i$ in $W^{1,p}(B_{2r})$ and $v_i\to v$ pointwise uniformly in $B_{2r}$, 
as $i\to\infty$ \cite[Theorem 4.1 (ii), p.~146]{EvaG92}.
By the Sobolev--Poincar\'e embedding $W^{1,p}(\partial B_R)\to C^{0, 1- \frac{n-1}{p}}(\partial B_R)$,
\[
\Big(\osc_{\partial B_R}v_i \Big)^p \lesssim R^{p-n+1} \int_{\partial B_R} |\nabla v_i|^p\,dS
\]
for every $R\in(0,2r)$, where $dS$ denotes the $(n-1)$-dimensional Hausdorff measure and the 
constant depends only on $p$ and $n$ (see, e.g.\ \cite[Lemma~1]{Geh62}, stated for the case $n=p=3$).
Integrating with respect to $R$ gives
\[
\int_r^{2r} \Big(\osc_{\partial B_R}v_i\Big)^p\,dR
\lesssim 
\int_r^{2r} R^{p-n+1} \int_{\partial B_R} |\nabla v_i|^p\,dS\,dR
\lesssim
r^{p-n+1}\int_{B_{2r}} |\nabla v_i|^p\,dx.
\]
Since $v_i \to v$ uniformly, we obtain that $(\osc_{\partial B_R}v)^p =\lim_{i\to\infty}(\osc_{\partial B_R}v_i)^p$ for every $R$.
Using this and $v_i \to v$ in $W^{1, p}(B_{2r})$,
it follows by Fatou's Lemma that 
\begin{align*}
\int_r^{2r}\Big(\osc_{\partial B_R}v\Big)^p\,dR
&\le \liminf_{i\to\infty}\int_r^{2r}\Big(\osc_{\partial B_R}v_i\Big)^p\,dR\\
& \le \liminf_{i\to\infty} Cr^{p-n+1}\int_{B_{2r}} |\nabla v_i|^p\,dx
= Cr^{p-n+1}\int_{B_{2r}} |\nabla v|^p\,dx.
\end{align*}
As $v$ is continuous and monotone in the sense of Lebesgue, we have that
$
\osc_{B_r}v
\le \osc_{B_R}v
= \osc_{\partial B_R}v
$
for $R\in(r,2r)$, and therefore
\[
r\Big(\osc_{B_r}v\Big)^p
\le \int_r^{2r}\Big(\osc_{\partial B_R}v\Big)^p\,dR
\le Cr^{p-n+1}\int_{B_{2r}} |\nabla v|^p\,dx
\le CrA.
\]
Since
\[
\osc_{B_r}v
=\sup_{x,y\in B_r}|v(x)-v(y)|
=\sup_{x,y\in B_r}\bigg|\log\frac{u(x)}{u(y)}\bigg|
= \log\frac{\sup_{B_r}u}{\inf_{B_r}u},
\]
it now follows that
\[
\sup_{B_r}u
\le \exp\big((CA)^{1/p}\big)\inf_{B_r}u. \qedhere
\]
\end{proof}

We conclude this section with the Harnack inequality for local minimizers. 
The novelty of the next theorem, apart from the technique, is that the constant depends only on $\frac qp$, not on 
$p$ and $q$ separately.

\begin{thm}[Harnack inequality]\label{thm:harnack}
Let $\phi \in \Phic(\Omega)$ satisfy \azero{} and \aone{}.
We assume that $B_{2r}\subset \Omega$ with $r\le 1$ and $\phi|_{B_{2r}}$ satisfies 
\ainc{p} and \adec{q} with $p,q\in (n, \infty)$.

Then any non-negative local minimizer $u \in W^{1, \phi}_\loc(\Omega)$ satisfies the Harnack inequality
\[
\sup_{B_r}(u+r)
\le C\inf_{B_r} (u+r)
\]
when $B_{4r}\subset \Omega$.
The constant $C$ depends only on $n$, $\beta$, $L_p$, $L_q$, $\frac qp$ and $\rho_\phi(|\nabla u|)$.
\end{thm}
\begin{proof}
Let $u \in W^{1, \phi}_\loc(\Omega)$ be a non-negative local minimizer.
By Theorem~\ref{thm:Bloch}, 
\[
\int_{B_r} |\nabla\log(u+r)|^n \,dx
\le C,
\]
where $C$ depends only on $n$, 
$L_p$, $L_q$, $\frac qp$, the constants from \azero{} and \aone{}, and $\rho_\phi(|\nabla u|)$.
Since $p>n$, $\phi$ satisfies \ainc{p} and $u\in W^{1, \phi}_\loc(\Omega)$, 
$u$ is continuous and Lemma~\ref{lem:monotone} yields that $u+r$ is monotone in the sense of Lebesgue. 
Thus we can apply Lemma~\ref{lem:Harnack} to $u+r$, which gives the Harnack inequality. 
%
\end{proof}

%
%


\section{Minimizers with non-doubling growth}

Let us study minimizers with given boundary values.

\begin{defn}
Let $p \in [1, \infty)$, $\phi\in \Phic(\Omega)$ and define, for $\lambda\ge 1$,
\[
\phi_\lambda(x,t)
:=
\int_0^t
\tfrac p\lambda \tau^{p-1} + \min\{\phi_-'(x,\tau), p\lambda \tau^{p-1}\}\, d\tau
=
\tfrac 1\lambda t^p + \int_0^t
\min\{\phi_-'(x,\tau), p\lambda \tau^{p-1}\}\, d\tau.
\]
\end{defn}

Note that since $t \mapsto \phi_\lambda(x,t)$ is convex for a.e.\ $x \in \Omega$, the left derivative $\phi_-'$ exists for a.e. $x \in \Omega$, and therefore the above definition makes sense.

\begin{lem}\label{lem:phi_lambda}
If $\phi\in\Phic(\Omega)$, then $\phi_\lambda\in \Phic(\Omega)$ 
satisfies $\phi_\lambda(\cdot,t)\approx t^p$ with constants depending on $\lambda$.
Furthermore, 
\[
\min\{\phi(x, \tfrac t2), \lambda (\tfrac t2)^p\} + \tfrac 1\lambda t^p 
\le 
\phi_\lambda(x,t) \le \phi(x, t) + \tfrac 1\lambda t^p
\] 
for $\lambda\ge 1$, $\phi_\lambda(x, t) \le \phi_\Lambda(x,t)+\tfrac1\lambda t^p$ for any $\Lambda\ge \lambda\ge 1$, and $\phi_\lambda\to \phi$ as $\lambda\to\infty$. 
\end{lem}
\begin{proof}
It follows from the definition that 
$\tfrac p\lambda \tau^{p-1} \le \phi_\lambda' (x,\tau)\le p(\tfrac1\lambda +\lambda)\tau^{p-1}$. 
Integrating over $\tau\in [0,t]$ gives $\phi_\lambda(\cdot,t)\approx t^p$.
Let $\Lambda\ge \lambda\ge 1$. Since the minimum in the integrand is increasing in $\lambda$, 
we see that 
\[
\phi_\lambda(x,t) 
\le 
\phi_\Lambda(x,t) + \big(\tfrac1\lambda-\tfrac1\Lambda\big)t^p
\le 
\phi(x,t) + \tfrac1\lambda t^p,
\]
and thus $\limsup_{\lambda\to\infty}\phi_\lambda\le \phi$. On the other hand, Fatou's Lemma 
gives 
\begin{align*}
\phi(x,t) 
&= 
\int_0^t
\lim_{\lambda\to\infty}(\tfrac p\lambda \tau^{p-1} + \min\{\phi_-'(x,\tau), p\lambda \tau^{p-1}\})\, d\tau \\
&\le
\liminf_{\lambda\to\infty}
\int_0^t
\tfrac p\lambda \tau^{p-1} + \min\{\phi_-'(x,\tau), p\lambda \tau^{p-1}\}\, d\tau
=
\liminf_{\lambda\to\infty} \phi_\lambda(x,t). 
\end{align*}
One of the terms in the minimum $\min\{\phi_-'(x,\tau), p\lambda \tau^{p-1}\}$ is achieved 
in at least a set of measure $\frac t2$. Since both terms are increasing in $\tau$, this implies that
\[
\phi_\lambda(x,t)
\ge
\min\bigg\{\int_0^{t/2}\phi_-'(x,\tau)\, d\tau,\int_0^{t/2}p\lambda \tau^{p-1}\, d\tau\bigg\}
+ \tfrac 1\lambda t^p 
=
\min\{\phi(x, \tfrac t2), \lambda (\tfrac t2)^p\} +\tfrac 1\lambda t^p. \qedhere
\]
\end{proof}

Since $\phi_\lambda(x,t)\approx t^p$, it follows by \cite[Proposition~3.2.4]{HarH19} that 
$W^{1,\phi_\lambda}(\Omega)=W^{1,p}(\Omega)$ and the norms $\|\cdot\|_{\phi_\lambda}$ and 
$\|\cdot\|_p$ are comparable. However, the embedding 
constant blows up as $\lambda\to \infty$ unless $\phi$ also satisfies \adec{p}. 
This approximation approach is similar to that in \cite{EleN_pp}.
Note in the next results that $f$ is bounded by the Sobolev embedding in 
$W^{1,p}(\Omega)$. 

\begin{lem}\label{lem:approximation}
Let $q:\Omega \to (n, \infty)$, and let $\phi\in\Phic(\Omega)$ satisfy \azero{}, \ainc{p} and \adec{\qx}, $p>n$.
Assume that $f\in W^{1,\phi}(\Omega)$ with $\varrho_\phi(\nabla f)<\infty$.
Then there exists a sequence $(u_{\lambda_k})$ of Dirichlet $\phi_{\lambda_k}$-energy 
minimizers with the boundary value function $f$ and a minimizer of the $\phi$-energy 
$u_\infty\in f+W^{1,\phi}_0(\Omega)$ such that 
 $u_{\lambda_k}\to u_\infty$ uniformly in $\Omega$ as $\lambda_k \to \infty$.
\end{lem}

\begin{proof}
Note that we use $W^{1,\phi_\lambda}(\Omega)=W^{1,p}(\Omega)$ and $W^{1,\phi_\lambda}_0(\Omega)=W^{1,p}_0(\Omega)$ several times in this proof.
Let $\lambda\ge p$. 
Note that $f\in W^{1,p}(\Omega)$ since $t^p\lesssim\phi(x,t)+1$ 
by \azero{} and \ainc{p}. 
By \cite[Theorem~6.2]{HarH19a} there exists a minimizer $u_\lambda\in f + W_0^{1,p}(\Omega)$ of
\[
\int_\Omega\phi_\lambda(x,|\nabla u|)\,dx.
\]

Fix $\lambda\ge 1$.
By Lemma~\ref{lem:phi_lambda} and $t^p\lesssim\phi(x,t)+1$, we have $t^p\lesssim 
\min\{\phi(x,\frac t2), \lambda t^p\}+1
\lesssim \phi_\lambda(x,t) +1$. Also by the same lemma, $\phi_\lambda\lesssim \phi + \tfrac{1}{\lambda} t^p\lesssim \phi +1$.
Since $f$ is a valid test-function and $u_\lambda$ is a $\phi_\lambda$-minimizer, we have
\[
\begin{split}
\int_\Omega |\nabla u_\lambda|^p\,dx
&\lesssim \int_\Omega\phi_\lambda(x,|\nabla u_\lambda|)+1\,dx
\le\int_\Omega\phi_\lambda(x,|\nabla f|)+1\,dx
\lesssim \int_\Omega\phi(x,|\nabla f|)+1\,dx
<\infty,
\end{split}
\]
 and hence $\rho_p(\nabla u_\lambda)$ is uniformly bounded. Note that the 
implicit constants do not depend on $\lambda$. 

Since $u_\lambda-f\in W^{1,p}_0(\Omega)$, the Poincar\'e inequality implies that 
\[
\|u_\lambda-f\|_p
\lesssim
\|\nabla (u_\lambda-f)\|_{p} 
\lesssim
\|\nabla u_\lambda\|_p 
+\|\nabla f\|_p 
\le c. 
\]
Therefore, 
$
\|u_\lambda\|_p 
\le \|u_\lambda-f\|_p + \|f\|_p
\le c
$
and so $\|u_\lambda\|_{1,p}$ is uniformly bounded.
Since $f + W_0^{1,p}(\Omega)$ is a closed subspace of $W^{1,p}(\Omega)$, it is a reflexive Banach space. Thus there exists a sequence 
$(\lambda_k)_{k=1}^\infty$ tending to infinity and a function $u_\infty \in f + W_0^{1,p}(\Omega)$ such that 
$u_{\lambda_k}\rightharpoonup u_\infty$ in $W^{1,p}(\Omega)$.
Since $p>n$, the weak convergence 
$u_{\lambda_k}-f\rightharpoonup u_\infty-f$ in $W^{1,p}_0(\Omega)$ and compactness of 
the Sobolev embedding \cite[Theorem 6.3 (Part IV), p.~168]{AdaF03}
imply that $u_{\lambda_k}-f \to u_\infty-f$ in the supremum norm. 
Hence $u_{\lambda_k} \to u$ uniformly in $\Omega$. 

We note that the modular $\rho_{\phi_\lambda}$ satisfies the conditions of \cite[Definition~2.1.1]{DieHHR11}.
Hence, it is weakly lower semicontinuous by \cite[Theorem~2.2.8]{DieHHR11}, and we obtain that
\begin{equation}\label{eq:u-infinity}
\begin{split}
\int_\Omega\phi_\lambda(x,|\nabla u_\infty|)\,dx
&\le \liminf_{k\to\infty} \int_\Omega\phi_\lambda(x,|\nabla u_{\lambda_k}|)\,dx
\le \liminf_{k\to\infty} \int_\Omega (1+\tfrac C{\lambda_k})\phi_{\lambda_k}(x,|\nabla u_{\lambda_k}|) + \tfrac C{\lambda_k}\,dx\\
&\le \liminf_{k\to\infty} \int_\Omega\phi_{\lambda_k}(x,|\nabla u_{\lambda_k}|) \,dx
\lesssim \int_\Omega \phi(x,|\nabla f|)+1\,dx
\end{split}
\end{equation}
for fixed $\lambda\ge 1$, where in the second inequality we used Lemma~\ref{lem:phi_lambda} and the fact that $t^p \lesssim \phi_\lambda(x,t) +1$.
It follows by monotone convergence that
\[
\int_\Omega\phi(x,|\nabla u_\infty|)\,dx
=
\lim_{\lambda\to\infty} \int_\Omega \min\{\phi(x,|\nabla u_\infty|), \lambda |\nabla u_\infty|^p\}\,dx
\le 
\limsup_{\lambda\to\infty}\int_\Omega \phi_\lambda(x,|\nabla u_\infty|)\,dx,
\]
and hence $|\nabla u_\infty| \in L^\phi(\Omega)$.
Since $p>n$, $\Omega$ is bounded and $u_\infty -f \in W^{1,p}_0(\Omega)$, we obtain by \cite[Theorem 2.4.1, p.~56]{Zie89}
that $u_\infty - f\in L^\infty(\Omega)$.
Moreover, $L^\infty(\Omega)\subset L^\phi(\Omega)$ since 
$\Omega$ is bounded and $\phi$ satisfies \azero{}. 
These and $f\in L^\phi(\Omega)$ yield that $u_\infty \in L^{\phi}(\Omega)$. 
Hence we have $u_\infty \in W^{1,\phi}(\Omega)$.
Since $u_\infty -f \in W^{1,p}_0(\Omega)$ and
$p>n$, it follows that $u_\infty-f$ can be continuously extended by $0$ in $\Omega^c$ \cite[Theorem~9.1.3]{AdaH96}. 
Then we conclude as in \cite[Lemma~1.26]{HeiKM06} that $u_\infty-f\in W^{1,\phi}_0(\Omega)$.

We conclude by showing that $u_\infty$ is a minimizer. Suppose to the contrary that there exists $u\in f+W^{1,\phi}_0(\Omega)$ with 
\[
\int_\Omega\phi(x,|\nabla u_\infty|)\,dx - \int_\Omega\phi(x,|\nabla u|)\,dx =:\epsilon > 0.
\]
By $\phi_\lambda\lesssim \phi+1$, $\phi_\lambda\to\phi$ and dominated convergence, there exists $\lambda_0$ such that 
\[
\int_\Omega\phi_\lambda(x,|\nabla u_\infty|)\,dx - \int_\Omega\phi_\lambda(x,|\nabla u|)\,dx \ge\tfrac\epsilon2 
\]
for all $\lambda\ge \lambda_0$. From the lower-semicontinuity estimate \eqref{eq:u-infinity} 
we obtain $k_0$ such that 
\[
\int_\Omega\phi_\lambda(x,|\nabla u_\infty|)\,dx
\le
\int_\Omega\phi_{\lambda_k}(x,|\nabla u_{\lambda_k}|)\,dx + \tfrac\epsilon4
\]
for all $k\ge k_0$. By increasing $k_0$ if necessary, we may assume that $\lambda_k\ge \lambda_0$ when 
$k\ge k_0$. For such $k$ we choose $\lambda=\lambda_k$ above and obtain that  
\[
\int_\Omega\phi_{\lambda_k}(x,|\nabla u|) + \tfrac\epsilon2
\le
\int_\Omega\phi_{\lambda_k}(x,|\nabla u_{\lambda_k}|)\,dx + \tfrac\epsilon4.
\]
This contradicts $u_{\lambda_k}$ being a $\phi_{\lambda_k}$-minimizer, since $u \in f + W^{1, \phi}_0(\Omega) \subset f + W^{1, \phi_{\lambda_k}}_0(\Omega)$.
Hence the counter-assumption was incorrect, and the minimization property of $u_\infty$ is proved. 
\end{proof}
%
We conclude this paper with the Harnack inequality for $\phi$-harmonic functions.
Here we use Remark~\ref{rem:caccioppoliHole} to handle the possibility that 
$q$ could be unbounded and thus $\phi$ non-doubling, like in Example~\ref{eg:varExp}. 
This is possible since $q^\circ$ is the supremum of $q$ only in the annulus, not the whole ball.

\begin{thm}[Harnack inequality]\label{thm:harnack-2}
Let $p, q : \Omega \to (n, \infty)$, and $\phi \in \Phic(\Omega)$ be strictly convex and satisfy \azero{}, \aone{}, \ainc{\px} and \adec{\qx}
with $\inf p>n$. 
Assume that $f\in W^{1,\phi}(\Omega)$ with $\rho_\phi(|\nabla f|)<\infty$.
Then there exists a unique minimizer $u$ of the $\phi$-energy with boundary values $f$.
Let $B_{4r}\subset \Omega$, $\displaystyle p^-:=\inf_{B_{2r}} p$ and $\displaystyle q^\circ:=\sup_{B_{2r}\setminus B_r} q$.
If $\frac{q^\circ}{p^-}<\infty$, then the Harnack inequality
\[
\sup_{B_r}(u+r)
\le C\inf_{B_r} (u+r)
\]
holds for all non-negative minimizers with $C$ depending only on $n$, $\beta$, $L_p$, $L_q$, $\frac{q^\circ}{p^-}$ and $\rho_\phi(|\nabla f|)$.
\end{thm}
\begin{proof}
By Lemma~\ref{lem:approximation}, there exists a sequence $(u_k)\subset f+W^{1,p}_0(\Omega)$ of 
minimizers of the $\phi_{\lambda_k}$ energy which converge uniformly to a minimizer 
$u_\infty\in f+W^{1,\phi}_0(\Omega)$ of the $\phi$-energy. Since $\phi$ is strictly convex, 
the minimizer is unique and so $u=u_\infty$. 

From \azero{} and \ainc{p^-} we conclude that $t^{p^-}\lesssim\phi(x,t)+1$. 
It follows from Lemma~\ref{lem:phi_lambda} that $\phi_\lambda(\cdot,t)\simeq \phi(\cdot, t)+\frac1\lambda t^{p^-}$. Thus $\phi_\lambda$ satisfies \azero{} and \aone{} with the same constants as $\phi$. 
Since $u_k\to u$ in $L^\infty(\Omega)$ and $u$ is non-negative we can choose a sequence 
$\epsilon_k\to 0^+$ such that $u_k+\epsilon_k$ is non-negative. 
By Theorem~\ref{thm:Bloch} with Remark~\ref{rem:caccioppoliHole}, 
\[
\int_{B_r} |\nabla\log(u_k+\epsilon_k+r)|^n \,dx \le C,
\]
where $C$ depends only on $n$, 
$L_p$, $L_q$, $\frac{q^\circ}{p^-}$, $\beta$ from \azero{} and \aone{}, and 
$\rho_{\phi_{\lambda_k}}(|\nabla u_k|)$. Since $u_k$ is a minimizer, 
$\rho_{\phi_{\lambda_k}}(|\nabla u_k|)\le \rho_{\phi_{\lambda_k}}(|\nabla f|)\lesssim 
\rho_\phi(|\nabla f|)+1$. 
Thus by Lemma~\ref{lem:Harnack}, we have
\[
\sup_{B_r}(u_k+\epsilon_k+r)
\le C\inf_{B_r} (u_k+\epsilon_k+r),
\]
with $C$ independent of $k$.
Since $u_k+\epsilon_k\to u_\infty$ uniformly, the claim follows.
\end{proof}


\section*{Acknowledgment}

Peter H\"ast\"o was supported in part by the Jenny and Antti Wihuri Foundation.

\section*{Conflict of interest}
The authors  declare no conflict of interest.


%
%
%
%
%
%

\end{document}